\documentclass[12pt]{amsart}

\usepackage{latexsym}
\usepackage{amssymb}
\usepackage[cp850]{inputenc}
\usepackage{epsfig}
\usepackage{psfrag}
\usepackage{amsthm}
\usepackage{amscd}
\usepackage{amsmath}
\usepackage{amsfonts}
\usepackage{graphics}
\usepackage[all]{xy}

\newtheorem{sat}{Theorem}				\newtheorem{lem}{Lemma}

\newtheorem*{defi*}{Definition}	\newtheorem*{bei*}{Example}
\newtheorem*{sat*}{Theorem}			\newtheorem*{kor*}{Corollary}
\newtheorem*{rmk*}{Remark}

\let\ssection=\section
\renewcommand{\section}{\setcounter{equation}{0}\ssection}

\newtheorem*{namedtheorem}{\theoremname}
\newcommand{\theoremname}{testing}
\newenvironment{named}[1]{\renewcommand{\theoremname}{#1}\begin{namedtheorem}}{\end{namedtheorem}}

\theoremstyle{remark}
\newtheorem*{bem}{Remark}

\newcommand{\BR}{\mathbb R}			
			
			\newcommand{\BZ}{\mathbb Z}
			\newcommand{\BT}{\mathbb T}

\newcommand{\CA}{\mathcal A}

\newcommand{\CS}{\mathcal S}		\newcommand{\CT}{\mathcal T}
		
		\newcommand{\CX}{\mathcal X}
\newcommand{\CY}{\mathcal Y}

\newcommand{\D}{\partial}

\DeclareMathOperator{\Diff}{Diff}	
\DeclareMathOperator{\SL}{SL}		

\DeclareMathOperator{\Id}{Id}		
\DeclareMathOperator{\vol}{vol}		

\DeclareMathOperator{\rank}{rank}

\DeclareMathOperator{\SO}{SO}

\newcommand{\trans}{{ }^t}
\newcommand{\bs}{\backslash}
\DeclareMathOperator{\syst}{syst}
\DeclareMathOperator{\vcdim}{vcdim}

\DeclareMathOperator{\cdim}{cdim}

\begin{document}

\title[]{The spine which was no spine}
\author{Alexandra Pettet \& Juan Souto}

\begin{abstract}
Let $\CT_n$ be the Teichm\"uller space of flat metrics on the $n$-dimensional torus $\BT^n$ and identify $\SL_n\BZ$ with the corresponding mapping class group.  We prove that the subset $\CY$ consisting of those points at which the systoles generate $\pi_1(\BT^n)$ is, for $n \geq 5$, not contractible. In particular, $\CY$ is not a $\SL_n\BZ$-equivariant deformation retract of $\CT_n$.
\end{abstract}
\maketitle

For $n\ge 2$ let $\CT_n$ be the Teichm\"uller space of flat metrics with unit volume on the $n$-dimensional torus $\BT^n=\BR^n/\BZ^n$. To be more precise, $\CT_n$ is the set of equivalence classes of unit volume flat metrics on $\BT^n$ where two metrics $\rho$ and $\rho'$ are equivalent if there is an orientation preserving diffeomorphism $\phi\in\Diff_+(\BT^n)$ homotopic to the identity with $\rho'=\phi^*\rho$. We consider on the Teichm\"uller space $\CT_n$ the topology with respect to which the classes of two flat metrics $\rho$ and $\rho'$ are close if there is a diffeomorphism $\phi\in\Diff_+(\BT^n)$ homotopic to the identity such that $\rho'$ and $\phi^*\rho$ are close as tensors.

Every element $A\in\SL_n\BZ$ induces an orientation preserving diffeomorphism $A\in\Diff_+(\BT^n)$ which is said to be {\em linear}. We obtain thus a right action of $\SL_n\BZ$ on $\CT_n$:
$$\CT_n\times\SL_n\BZ\to\CT_n,\ (\rho,A)\mapsto A^*\rho$$
which is properly discontinuous. There exists a finite index subgroup $\Gamma$ of $\SL_n\BZ$ which acts freely; in particular, the contractibility of $\CT_n$ implies that for any such subgroup $\Gamma$, the quotient $\CT_n/\Gamma$ is an Eilenberg-MacLane space for $\Gamma$.

The systole $\syst(\rho)$ of a point $\rho\in\CT_n$ is the length of the shortest homotopically essential geodesic in the flat torus $(\BT^n,\rho)$. Let $\CS(\rho)$ be the set of homotopy classes of geodesics in $(\BT^n,\rho)$ with length $\syst(\rho)$; the elements in $\CS(\rho)$ are known as the {\em systoles} of $(\BT^n,\rho)$. Ash \cite{Ash-Morse} proved that the systole function
$$\CT_n\to(0,\infty),\ \ \rho\mapsto\syst(\rho)$$
is a $\SL_n\BZ$-equivariant topological Morse function, and so it is not surprising that it can be used to construct a particularly nice $\SL_n\BZ$-equivariant spine, i.e., deformation retract, of $\CT_n$. More precisely, the following result was proved in a different language and much greater generality by Ash \cite{Ash}:

\begin{sat}[Ash]\label{well-rounded}
The subset $\CX$ of $\CT_n$ consisting of those points $\rho$ with the property that $\CS(\rho)$ generates a finite index subgroup of $\pi_1(\BT^n)$ is an $\SL_n\BZ$-equivariant spine of $\CT_n$.
\end{sat}

From a geometric point of view, that the systoles generate a finite index subgroup of $\pi_1(\BT^n)$ seems to be a peculiar condition. This led the authors to wonder whether the subset $\CY$ of $\CT_n$ consisting of those points $\rho\in\CT_n$ with the property that the systoles generate the full group $\pi_1(\BT^n)$ could be a $\SL_n\BZ$-equivariant deformation retract as well. For $n=2$, $3$ and $4$, this is known, as for these cases the sets $\CX$ and $\CY$ coincide \cite{lattices,van der Waerden}. The goal of this note is to show that this fails to be true for $n\ge 5$, although the complex $\CY$ is always a CW-complex of dimension $\frac{n(n-1)}{2}$.

\begin{sat}\label{main}
For $n\ge 5$, the subset $\CY$ of $\CT_n$ consisting of those points $\rho$ with the property that $\CS(\rho)$ generates $\pi_1(\BT^n)$ is not contractible and hence it is not a $\SL_n\BZ$-equivariant spine.
\end{sat}

Observe that Ash's spine $\CX$, known as {\em the well-rounded retract}, is homeomorphic to a CW-complex with the same dimension as the virtual cohomological dimension $\vcdim(\SL_n\BZ)=\frac{n(n-1)}2$ of $\SL_n\BZ$. The complex $\CY$ is also a CW-complex of the correct dimension.
\vspace{0.2cm}

In order to prove Theorem \ref{main}, we make use of the well-known identification between the Teichm\"uller space $\CT_n$ and the symmetric space $S_n=\SO_n\bs\SL_n\BR$. We discuss this identification in Section \ref{generalities}. For the convenience of the reader, we also sketch briefly the proof of Theorem \ref{well-rounded} in Section \ref{retract}. Now let $\Gamma$ be a torsion free finite index subgroup of $\SL_n\BZ$. The action of $\Gamma$ on $S_n$ is free and hence the quotient $M_\Gamma=S_n/\Gamma$ is a manifold. Borel and Serre \cite{Borel-Serre} constructed a compact manifold $\bar M_\Gamma$ with boundary $\D\bar M_\Gamma$ whose interior is homeomorphic to $M_\Gamma$. In section \ref{topology} we briefly describe how to construct non-trivial homology classes  in $H_{\frac{n(n-1)}2}(M_\Gamma)$ and $H_{n-1}(\bar M_\Gamma,\D\bar M_\Gamma)$. These classes are then used in Section \ref{sec:proof} to show that whenever $\Gamma$ is as above and is contained in the kernel of the standard homomorphism $\SL_n\BZ\to\SL_n\BZ/2\BZ$, the inclusion $\CY/\Gamma\to M_\Gamma$ is not surjective on the $\frac{n(n-1)}2$-homology; Theorem \ref{main} follows. 
\vspace{0.2cm}

We thank Martin Henk for showing us an example of a point $\CX\setminus\CY$. We also thank Mladen Bestvina for convincing us that there was no way that $\CY$ was a retract, and for almost completely proving it for us. The second author is grateful to the Department of Mathematics of Stanford University for its hospitality while this note was being written.

%
%
%
%

\section{Generalities}\label{generalities}

We begin by fixing some notation that will be used in the sequel. We denote by $\{e_1,\dots,e_n\}$ and $\langle\cdot,\cdot\rangle$ the standard basis and scalar product on $\BR^n$. If $v$ or $A$ are vectors or matrices we let $\trans v$ and $\trans A$ denote their transposes. Using this notation $\vert v\vert=\sqrt{\trans vv}$ is the standard euclidean norm on $\BR^n$. If $\CS$ is a subset of a group then we denote by $\langle\CS\rangle$ the subgroup generated by $\CS$; for example, $\BZ^n=\langle\{e_1,\dots,e_n\}\rangle$. If $\CS$ is a subset of a euclidean vector space, we denote by $\langle\CS\rangle_\BR$ the $\BR$-linear subspace generated by $\CS$ and by 
$\langle\CS\rangle_\BR^\bot$ its orthogonal complement. We will sometimes use the same symbol to denote both an equivalence class and a representative of the equivalence class. For example, we may use the same notation for an element in $\SL_n\BR$, and the corresponding element in the symmetric space $S_n=\SO_n\bs\SL_n\BR$ or in the even smaller quotient $S_n/\SL_n\BZ$. When we do want to distinguish the class of $A$,  we denote it by $[A]$, and we will consistently denote the homology class corresponding to a cycle $\beta$ by $[\beta]$. All the homology groups considered below will have coefficients in the field $\BZ/2\BZ$ of two elements.
\vspace{0.2cm}

These platitudes out of the way, we recall briefly the identification between the Teichm\"uller space $\CT_n$ and the symmetric space $S_n=\SO_n\bs\SL_n\BR$.  If $\rho$ is a flat metric on $\BT^n=\BR^n/\BZ^n$ with unit volume $\vol(\BT^n,\rho)=1$, the universal cover $\BR^n$ is a complete flat manifold with respect to the induced metric $\tilde\rho$. In particular, there is an orientation preserving isometry
$$\phi:(\BR^n,\tilde\rho)\to (\BR^n,\langle\cdot,\cdot\rangle)$$
The action by deck-transformations of the fundamental group $\pi_1(\BT^n)$ on $(\BR^n,\tilde\rho)$ is isometric. Conjugating this action by $\phi$ we obtain an action of $\pi_1(\BT^n)=\BZ^n$ on $(\BR^n,\langle\cdot,\cdot\rangle)$, also by isometries. It follows from a classical result of Bieberbach \cite{Wolf} that the group $\phi\pi_1(\BT^n)\phi^{-1}$ is a group of translations of $\BR^n$. In other words, the isometry $\phi$ induces a homomorphism
$$\BZ^n\to\BR^n,\ \ \gamma\mapsto\{x\mapsto(\phi\circ\gamma\circ\phi^{-1})(x)\}$$
with discrete and co-compact image. Any such homomorphism is the restriction to $\BZ^n$ of an element in $\SL_n\BR$. Different choices for the isometry $\phi$ yield homomorphisms which differ by post-composition with an orthogonal transformation of $(\BR^n,\langle\cdot,\cdot\rangle)$, and hence elements in $\SL_n\BR$ which differ by left-multiplication with an element in $\SO_n$. Thus, to every flat metric on $\BT^n$ we can associate a well-defined point in the symmetric space $S_n=\SO_n\bs\SL_n\BR$. Moreover, equivalent flat metrics on $\BT^n$ induce the same point in $S_n$. We have thus a well-defined map
\begin{equation}\label{identification}
\CT_n\to S_n=\SO_n\bs\SL_n\BR
\end{equation}
The map \eqref{identification} is a homeomorphism. Observe that under the identification \eqref{identification}, the action of $\SL_n\BZ$ on $\CT_n$ corresponds to the action on $S_n$ by right multiplication.

As defined in the introduction, the systole $\syst(\rho)$ of a point $\rho\in\CT_n$ is the length of the shortest non-trivial geodesic in $(\BT^n,\rho)$ and $\CS(\rho)$ is the set of shortest non-trivial geodesics. Under the identification \eqref{identification}, for $A\in\SL_n\BR$ we have
$$\syst(A)=\min_{v\in\BZ^n,v\neq 0}\vert Av\vert$$
and
$$\CS(A)=\{v\in\BZ^n, \vert Av\vert=\syst(A)\}$$
In particular, Ash's well rounded spine $\CX$ and the complex $\CY$ considered in Theorem \ref{main} are given by:
\begin{align*}
\CX
&=\{\rho\in\CT_n\vert \langle\CS(\rho)\rangle\ \hbox{has finite index in}\ \pi_1(\BT^n)\}\\
&=\{A\in S_n\vert \langle\CS(A)\rangle\ \hbox{has finite index in}\ \BZ^n\}\\
\CY
&=\{\rho\in\CT_n\vert \langle\CS(\rho)\rangle=\pi_1(\BT^n)\}\\
&=\{A\in S_n\vert \langle\CS(A)\rangle=\BZ^n\}
\end{align*}

As was also mentioned in the introduction, Ash \cite{Ash-Morse} proved that the systole function
$$\CT_n\to(0,\infty),\ \ \rho\mapsto\syst(\rho)$$
is an $\SL_n\BZ$-equivariant topological Morse function. Here we will only use that the systole function is proper when considered as a function on $S_n/\SL_n\BZ$. 

\begin{named}{Mahler's compactness theorem}
For every $\epsilon>0$, the set of those $A\in S_n/\SL_n\BZ$ with $\syst(A)\ge\epsilon$ is compact.
\end{named}

Computations are simpler with matrices than with flat metrics, and so in the sequel we will mainly work in the symmetric space $S_n$. 

%
%
%
%

\section{The well-rounded retract}\label{retract}
In this section we discuss briefly the proof of Theorem \ref{well-rounded}. See \cite{Ash} for a complete proof of a more general version of this theorem.

\begin{named}{Theorem \ref{well-rounded}}[Ash]
The subset $\CX$ of $\CT_n$ consisting of those points $\rho$ with the property that $\CS(\rho)$ generates a finite index subgroup of $\pi_1(\BT^n)$ is an $\SL_n\BZ$-equivariant spine of $\CT_n$.
\end{named}

Recall that given $\rho\in\CT_n$ we denote by $\langle\CS(\rho)\rangle$ the subgroup $\pi_1(\BT^n)$ generated by the shortest non-trivial geodesics in $(\BT^n,\rho)$. Identifying $\pi_1(\BT^n)$ with $\BZ^n$ we see that the subgroup $\langle\CS(\rho)\rangle$ is a free abelian group with rank in $\{1,\dots,n\}$. Moreover, $\rank\langle\CS(\rho)\rangle=n$ if and only if $\langle\CS(\rho)\rangle$ has finite index in $\pi_1(\BT^n)$. For $k=1,\dots,n$ consider the set $\CX_k$ of those points $\rho\in\CT_n$ for which we have $\rank\langle\CS(\rho)\rangle\ge k$.
We have thus the following chain of nested $\SL_n\BZ$-invariant subspaces:
$$\CX=\CX_n\subset\CX_{n-1}\subset\dots\subset\CX_1=\CT_n$$
In order to prove Theorem \ref{well-rounded} it suffices to show that for $k=1,\dots,n-1$ the space $\CX_{k+1}$ is an $\SL_n\BZ$-equivariant spine of $\CX_k$. In order to see that this is the case we use freely the identification \eqref{identification} discussed above between the Teichm\"uller space $\CT_n$ and the symmetric space $S_n=\SO_n\bs\SL_n\BR$. 

Under this identification, a point $A\in S_n$ belongs to $\CX_k\setminus\CX_{k+1}$ if and only if the set $\CS(A)$ generates a rank $k$ subgroup of $\BZ^n$. Equivalently, $\CS(A)$ generates a $k$-dimensional $\BR$-linear subspace $\langle\CS(A)\rangle_\BR$ of $\BR^n$. Given $A\in\CX_k$ and $\lambda\in\BR$, consider the one-parameter family of linear maps
\begin{equation}\label{flow}
T_A^\lambda\in\SL_n\BR,\ \ \ T_A^\lambda(v)=\left\{
\begin{array}{ll}
e^{(n-k)\lambda}v  & \hbox{for}\ v\in A\langle\CS(A)\rangle_\BR  \\
e^{-k\lambda}v  &  \hbox{for}\ v\in (A\langle\CS(A)\rangle_\BR)^\bot
\end{array}
 \right.
\end{equation}
where $(A\langle\CS(A)\rangle_\BR)^\bot$ is the orthogonal complement in $(\BR^n,\langle\cdot,\cdot\rangle)$ of the image under $A$ of $\langle\CS(A)\rangle_\BR$. 

Now $T_A^0 A=A$, and if $A\in\CX_k\setminus\CX_{k+1}$, there is some $\lambda$ positive with $T_A^\lambda A\in\CX_{k+1}$. For $A\in\CX_k$ let $\tau(A)\ge 0$ be maximal such that
$$T_A^\lambda A\in\CX_k\setminus\CX_{k+1}\ \ \hbox{for all}\ \lambda\in(0,\tau(A))$$
By definition $\tau(A)=0$ for $A\in\CX_{k+1}$. The function $A\mapsto\tau(A)$ is continuous on $\CX_k$, which implies that 
\begin{equation}\label{homotopy}
[0,1]\times\CX_k\to\CX_k,\ \ (t,A)\mapsto T_A^{t\tau(A)}A
\end{equation}
is continuous as well. By definition, this homotopy is $\SL_n\BZ$-equivariant, starts with the identity, and ends with a projection of $\CX_k$ to $\CX_{k+1}$. This proves that $\CX_{k+1}$ is an $\SL_n\BZ$-equivariant spine of $\CX_k$ for $k=1,\dots,n-1$, concluding the sketch of the proof of Theorem \ref{well-rounded}.

\begin{bem}
Something must be done to verify the continuity of \eqref{homotopy} as the map 
$$\BR\times\CX_k\to\SL_n\BR,\ \ (\lambda,A)\mapsto T_A^\lambda A$$
itself is not continuous. The key point is that this map is continuous on $\BR\times(\CX_k\setminus\CX_{k+1})$, and by definition $\tau(A)=0$ for $A\in\CX_{k+1}$.
\end{bem}

We conclude this section with a couple of additional remarks about the structure of the well-rounded retract $\CX$ and a computation of the virtual cohomological dimension of $\SL_n\BZ$. 

It is not difficult to prove that $\CX_k$ is a co-dimension $k-1$ semi-algebraic set given by a locally finite collection of inequalities and quadratic algebraic equations. Hence $\CX$ is homeomorphic to a CW-complex of dimension 
$$\dim(\CX)=\dim S_n-(n-1)=\frac{n(n-1)}2$$
It is also easy to see that the well-rounded retract $\CX$ is cocompact, although $\CX_k$ is not cocompact for $k<n$.

The symmetric space $S_n$ is contractible, hence so is $\CX$. In particular, if $\Gamma$ is a subgroup of $\SL_n\BZ$ which acts freely on $S_n$, then $\CX/\Gamma$ is an Eilenberg-MacLane space for $\Gamma$, giving us the following upper bound on its cohomological dimension:
$$\cdim(\Gamma)\le\dim(X)=\frac{n(n-1)}2$$
The group $\SL_n\BZ$ contains subgroups $\Gamma$ of finite index which are torsion free and thus act freely on $S_n$. This yields the upper bound
$$\vcdim(\SL_n\BZ)\le \frac{n(n-1)}2$$
for the virtual cohomological dimension of $\SL_n\BZ$. One can see the upper bound is sharp as follows: Let $N$ be the $\frac{n(n-1)}2$-dimensional subgroup of $\SL_n\BR$ consisting of upper triangular matrices with units in the diagonal. The intersection $N\cap\SL_n\BZ$ is a cocompact subgroup of $N$; hence for $\Gamma$ as above $N/(N\cap\Gamma)$ is a closed manifold of dimension $\frac{n(n-1)}2$. The group $N$ is contractible, hence $N/(N\cap\Gamma)$ is an Eilenberg-MacLane space for $N\cap\Gamma$. Thus we have 
$$\cdim(\Gamma)\ge\cdim(N\cap\Gamma)=\dim(N/(N\cap\Gamma))=\frac{n(n-1)}2$$
This implies that $\vcdim(\SL_n\BZ)=\frac{n(n-1)}2$.

In the next section we will give an elementary argument to prove that the homology class $[N/(N\cap\Gamma)]\in H_{\frac{n(n-1)}2}(M_\Gamma)$ is non-trivial.

%
%
%
%

\section{Some topology}\label{topology}
As mentioned some lines above, $\SL_n\BZ$ contains a torsion free subgroup of finite index, and any such subgroup acts not only discretely, but also freely on $S_n$; hence the quotient $M_\Gamma=S_n/\Gamma$ is a manifold. Borel and Serre \cite{Borel-Serre} proved that $M_\Gamma$ is homeomorphic to the interior of a compact manifold $\bar M_\Gamma$ with boundary $\D\bar M_\Gamma$. Identifying $\bar M_\Gamma$ with the complement of an open regular neighborhood of $\D\bar M_\Gamma$ we consider from now on the former as a submanifold of $M_\Gamma$ and choose a map
\begin{equation}\label{projection}
p:M_\Gamma\to\bar M_\Gamma
\end{equation}
whose restriction to $\bar M_\Gamma$ is the identity. 

\begin{bem}
Grayson \cite{Grayson} gave a construction of $\bar M_\Gamma$ directly as a submanifold of $M_\Gamma$, giving a new proof of some of Borel's and Serre's results. If we are only interested in constructing a compactification $\bar M_\Gamma$ as above, we can do the following: For $A\in\SL_n\BR$ the series $\sum_{v\in\BZ^n}e^{-\vert Av\vert}$ converges, and its value depends only on the class of $A$ in $S_n$. In particular, the function 
$$F:S_n\to\BR, \ \ F(A)=\sum_{v\in\BZ^n}e^{-\vert Av\vert}$$
is well-defined, smooth, and descends to a function $f:M_\Gamma\to\BR$. The function $f$ is proper, and there is some constant $L$ which bounds above the critical values of $f$. This implies that $f^{-1}[L,\infty)$ is a product, hence we can set $\bar M_\Gamma=f^{-1}[0,L]$.
\end{bem}

Borel and Serre constructed the compactification $\bar M_\Gamma$ to study homological properties of $\Gamma$. We will only need some basic facts, well-known probably to experts and non-experts alike, which we deduce in an elementary way. 

Recall that we always consider homology with coefficients in $\BZ/2\BZ$. By Lefschetz duality there is a non-degenerate pairing
$$\iota:H_{\frac{n(n-1)}2}(M_\Gamma)\times H_{n-1}(\bar M_\Gamma,\D\bar M_\Gamma)\to\BZ/2\BZ$$
which can be computed as follows. Given homology classes $[\alpha]\in H_{\frac{n(n-1)}2}(M_\Gamma)$ and $[\beta]\in H_{n-1}(\bar M_\Gamma,\D\bar M_\Gamma)$, represent them by cycles $\alpha$ and $\beta$ in general position. Then $\iota([\alpha],[\beta])$ is just the parity of the cardinality of the set $\alpha\cap\beta$.

\begin{bem}
This is the simplest version of the Alexander-Whitney product in homology, which dualizes the cup product.
\end{bem}

In particular, in order to prove that the $\frac{n(n-1)}2$-cycle $\alpha=N/(N\cap\Gamma)$ represents a non-trivial homology class it suffices to find a cycle $\beta\in C_{n-1}(\bar M_\Gamma,\D\bar M_\Gamma)$ which intersects $\alpha$ transversally at a single point. In order to find such a cycle $\beta$ we consider the subgroup $\Delta$ of $\SL_n\BR$ consisting of diagonal matrices with positive entries and the map $\Delta\to M_\Gamma$ which maps every $H\in\Delta$ to its class in $M_\Gamma=\SO_n\bs\SL_n\BR/\Gamma$. By Mahler's compactness theorem, the systole function is proper on $S_n/\SL_n\BZ$; since $\Gamma$ has finite index in $\SL_n\BZ$ it is also proper on $M_\Gamma$. Then the following lemma implies that the map $\Delta\to M_\Gamma$ is proper as well.

\begin{lem}\label{help-diag}
Let $H\in\Delta$ be a diagonal matrix with positive entries. Then $\syst(H)$ is the minimum of the entries in the diagonal of $H$. In particular $\syst(H)\le 1$, with equality if and only if $H=\Id$. 
\end{lem}
\begin{proof}
Let $a_1,\dots,a_n$ be the diagonal entries of $H$, and for the sake of concreteness assume that $a_1$ is minimal. Then for $v=\trans(v_1,\dots,v_n)\in\BZ^n$ with, say, $v_i\neq 0$, we have
$$\vert Av\vert=\sqrt{a_1^2v_1^2+\dots+a_n^2v_n^2}\ge\vert a_iv_i\vert\ge a_i\ge a_1$$
with equality if, for example, $v_1=1$ and $v_2=\dots=v_n=0$. This proves the first claim of the lemma. The second claim follows from the fact that $a_1\dots a_n=1$ so that either some $a_i$ is less than $1$ or all of the $a_i$'s are equal to $1$.
\end{proof}

Composing the proper map $\Delta\to M_\Gamma$ with the projection \eqref{projection} we obtain a cycle $\beta$ in $C_{n-1}(\bar M_\Gamma,\D\bar M_\Gamma)$. We denote by $[\Delta]=[\beta]$ the homology class of $\beta$.

\begin{lem}\label{help-nil}
Let $A\in N$ be an upper triangular matrix with $1$ at the diagonal. Then $\syst(A)=1$.
\end{lem}
\begin{proof}
Given $v=\trans(v_1,\dots,v_n)\in\BZ^n$, let $i$ be minimal such that $v_j=0$ for all $j>i$. Then we have that $v_i$ is the $i$-th coordinate of $Av$ and hence $\vert Av\vert\ge\vert v_i\vert\ge 1$, with equality when, for example, $v_1=1$ and $v_2=\dots=v_n=0$.
\end{proof}

The intersection points of the cycles $\alpha=N/(N\cap\Gamma)$ and $\beta$ in $M_\Gamma$ correspond bijectively to the set of those $H\in\Delta$ for which there is $A\in\Gamma$ with $HA\in N$.  For any such $H$ we have by Lemma \ref{help-nil}
$$1=\syst(HA)=\syst(H)$$
and hence $H=\Id$; thus $\alpha$ and $\beta$ intersect at a single point. Moreover, their intersection is locally modeled by the intersection of the images of $\Delta$ and $N$ in $S_n$ and hence it is transversal; therefore $\iota([\alpha],[\beta])=1$. This implies that $[\alpha]=[N/(N\cap\Gamma)]$ and $[\beta]=[\Delta]$ are not homologically trivial.

\begin{lem}\label{nil-essential}
If $\Gamma$ is a torsion-free subgroup of $\SL_n\BZ$ then the classes $[N/N\cap\Gamma]\in H_{\frac{n(n-1)}2}(M_\Gamma)$ and $[\Delta]\in H_{n-1}(\bar M_\Gamma,\D\bar M_\Gamma)$ have intersection
$$\iota([N/N\cap\Gamma],[\Delta])=1$$
and hence are not trivial.\qed
\end{lem}

%
%
%
%

\section{Proof of Theorem \ref{main}}\label{sec:proof}
Taking into account the title of this section, it can hardly be surprising that we now prove:

\begin{named}{Theorem \ref{main}}
For $n\ge 5$, the subset $\CY$ of $\CT_n$ consisting of those points $\rho$ with the property that $\CS(\rho)$ generates $\pi_1(\BT^n)$ is not contractible and hence it is not a $\SL_n\BZ$-equivariant spine.
\end{named}

Let all the notation be as in the previous section. As mentioned in the introduction, in order to prove Theorem \ref{main} we will show that there is a finite index torsion free subgroup $\Gamma\subset\SL_n\BZ$ for which the map 
\begin{equation}\label{map-homology}
H_{\frac{n(n-1)}2}(\CY/\Gamma)\to H_{\frac{n(n-1)}2}(M_\Gamma)
\end{equation}
is not surjective. More precisely, we will show that this is the case for those torsion-free finite index subgroups $\Gamma$ contained in the kernel of the homomorphism
\begin{equation}\label{congruence}
\SL_n\BZ\to\SL_n\BZ/2\BZ
\end{equation}
Fix such a $\Gamma$ and let $A\in\SL_n\BR$ be the upper triangular matrix which, up a factor, is the identity on the upper left $(n-1)\times(n-1)$ quadrant and with entries equal to $\frac 12$ in the last column 
\begin{equation}\label{point}
A=2^{-\frac 1n}\left(
\begin{array}{ccccc}
1  & 0  & \dots & 0 & \frac 12  \\
0  & 1  & \dots  & 0 & \frac 12 \\
\vdots  & \vdots  & \ddots  & \vdots & \vdots \\
0  & 0  & \dots  & 1 & \frac 12 \\
0  & 0  & \dots  & 0 & \frac 12
\end{array}
\right)
\end{equation}
The assumption that $\Gamma$ is contained in the kernel of \eqref{congruence} implies that every element $B\in\Gamma$ can be written as $B=\Id+B'$ where every entry of $B'$ is even. In particular, we have for any such $B$ that $ABA^{-1}$ has integer entries and hence that 
$$A\Gamma A^{-1}\subset\SL_n\BZ$$
Observe that we have a diffeomorphism $\CA:M_{A\Gamma A^{-1}}\to M_\Gamma$ such that the following diagram commutes:
$$\xymatrix{
S_n\ar[d]\ar[r]^{\{[B]\mapsto[BA]\}} & S_n\ar[d]  \\
M_{A\Gamma A^{-1}}\ar[r]^{\CA}& M_\Gamma
}$$
The diffeomorphism $\CA$ maps the non-trivial, by Lemma \ref{nil-essential}, homology classes 
$$[N/(N\cap(A\Gamma A^{-1}))]\in H_{\frac{n(n-1)}2}(M_{A\Gamma A^{-1}}), [\Delta]\in H_{n-1}(\bar M_{A\Gamma A^{-1}},\D\bar M_{A\Gamma A^{-1}})$$ 
to, a fortiori, non-trivial classes with
$$\iota(\CA_*[\Delta],\CA_*([N/(N\cap(A\Gamma A^{-1}))]))=1$$
Observe that the class $\CA_*[\Delta]\in H_{n-1}(\bar M_\Gamma,\D\bar M_{\Gamma})$ is represented by a cycle supported in $\{HA\vert H\in\Delta\}\cap\bar M_\Gamma$. Below we will prove

\begin{lem}\label{plane}
Assume that $n\ge 5$, that $A$ is the matrix given in \eqref{point} and that $H\in\Delta$ is a diagonal matrix. Then we have:
\begin{itemize}
\item $A\in\CX\setminus\CY$, and
\item $HA\in\CX$ if and only if $H=\Id$.
\end{itemize}
\end{lem}

Lemma \ref{plane} implies that the homologically non-trivial class $\CA_*[\Delta]$ is supported by a cycle which does not intersect $\CY/\Gamma$. This implies then that the class $\CA_*([N/(N\cap(A\Gamma A^{-1}))])\in H_{\frac{n(n-1)}2}(M_\Gamma)$ is not represented by any cycle in $C_{\frac{n(n-1)}2}(\CY/\Gamma)$. In particular, we deduce that as claimed \eqref{map-homology} is not surjective. We can now conclude the proof of Theorem \ref{main}. If $\CY$ were contractible, then $\CY/\Gamma$ would be an Eilenberg-MacLane space for $\Gamma$ and the inclusion $\CY/\Gamma\hookrightarrow S_n/\Gamma=M_\Gamma$ a homotopy equivalence, contradicting the lack of surjectivity of \eqref{map-homology}.
\vspace{0.2cm}

It just remains to prove Lemma \ref{plane}:

\begin{proof}[Proof of Lemma \ref{plane}]
We start proving that $A\in\CX\setminus\CY$. For every vector $v=\trans(v_1,\dots,v_n)\in\BZ^n$ we have that 
$$\trans(Av)=2^{-\frac 1n}\left(v_1+\frac {v_n}2,\dots,v_{n-1}+\frac {v_n}2,\frac{v_n}2\right)$$
If $v_n$ is odd, then $\vert Av\vert\ge\frac{\sqrt n}22^{-\frac 1n}$. On the other hand, if $v_n$ is even every vector has at least length $2^{-\frac 1n}$ with, for example, equality for $e_1$. This proves that $\syst(A)=2^{-\frac 1n}$ and one can easily see that $\CS(A)$ consists of the following $2n$ vectors in $\BZ^n$
$$\pm e_1,\dots,\pm e_{n-1},\pm (2e_n-\sum_{i=1}^{n-1}e_i)$$
This implies that $\CS(A)$ generates the subgroup of $\BZ^n$ consisting of vectors whose last coordinate is even. This is a proper subgroup with index 2, hence $A\notin\CY$ but $A\in\CX$.
\vspace{0.2cm}

Continuing with the proof of the lemma let $H\in\Delta$ be a diagonal matrix with positive entries $a_1,\dots,a_n$. When we multiply $H$ and $A$ we obtain:
\begin{equation}\label{ugly}
HA=2^{-\frac 1n}\left(
\begin{array}{ccccc}
a_1  & 0  & \dots & 0 & \frac {a_1}2  \\
0  & a_2  & \dots  & 0 & \frac {a_2}2 \\
\vdots  & \vdots  & \ddots  & \vdots & \vdots \\
0  & 0  & \dots  & a_{n-1} & \frac {a_{n-1}}2 \\
0  & 0  & \dots  & 0 & \frac {a_n}2
\end{array}
\right)
\end{equation}
For any such $HA$ and $i=1,\dots,n-1$ we have $\vert HAe_i\vert=2^{-\frac 1n}a_i$. We also have $\vert HA(2e_n-\sum_{i=1}^{n-1}e_i)\vert=2^{-\frac 1n}a_n$. This shows that
\begin{equation}\label{upper-bound}
\syst(HA)\le2^{-\frac 1n}\min\{a_i\vert i=1,\dots,n\}
\end{equation}
Assume from now on that $HA$ belongs to the well-rounded retract $\CX$ and recall that this means that the set $\CS(HA)$ of those $v\in\BZ^n$ with $\vert HAv\vert=\syst(HA)$ generates a finite index subgroup of $\BZ^n$. In particular, there is a shortest vector $v=\trans(w_1,\dots,w_n)\in\CS(HA)$ with $w_n>0$. For such a $v$ one has
$$\syst(HA)=\vert HAv\vert\ge2^{-\frac 1n}\frac{w_n}2a_n$$
We deduce then from \eqref{upper-bound} that $w_n$ is either $1$ or $2$. We claim that $w_n=2$. Otherwise one has
$$\vert HAv\vert\ge\frac 12\sqrt{a_1^2+\dots+a_{n-1}^2+a_n^2}\ge2^{-\frac 1n}\frac{\sqrt n}2\min\{a_i\vert i=1,\dots,n\}$$
contradicting \eqref{upper-bound}, as $n\ge 5$. Hence there is a shortest vector with last coefficient $w_n=2$. Among all these vectors, $HAv$ is minimal if and only if $v=2e_n$; thus $\syst(HA)=2^{-\frac 1n}a_n$. The assumption that $HA\in\CX$ implies that for $i=1,\dots,n-1$ there is also some vector $v'$ with $\vert HAv'\vert=\syst(HA)=2^{-\frac 1n}a_n$ and whose $i$-th coefficient $w_i'$ does not vanish. By the discussion above, the last coefficient of $v'$ must vanish and hence the $i$-th coefficent of $HAv$ is $2^{-\frac 1n}w_i'a_i$. This implies that $a_i=a_n$. We have proved that if $HA\in\CX$ then $H=\Id$.
\end{proof}

%
%
%
%

{\sc \tiny \noindent
Alexandra Pettet, Department of Mathematics, Stanford University}
\vspace{0.2cm}

{\sc \tiny \noindent
Juan Souto, Department of Mathematics, University of Chicago}

\end{document}